\documentclass{amsart}

\usepackage{enumerate}
\usepackage{amsmath}
\usepackage{mathrsfs}
\usepackage{color}
\usepackage{amssymb}

\usepackage[colorlinks, linkcolor=blue, citecolor=blue]{hyperref}

\newtheorem{theorem}{Theorem}[section]
\newtheorem{lemma}[theorem]{Lemma}
\newtheorem{corollary}[theorem]{Corollary}

\newtheorem{remark}[theorem]{Remark}
\newtheorem{definition}[theorem]{Definition}

\makeatletter
\@addtoreset{equation}{section}
\makeatother

\begin{document}
\title[Noncommutative 2-Co-lacunary sequences]{ 2-Co-lacunary sequences in noncommutative symmetric Banach spaces}

\author[F. Sukochev]{Fedor Sukochev}
\address{Fedor Sukochev \\ School of Mathematics and Statistics,
 University of New South Wales, Kensington 2052, Australia}
\email{f.sukochev@unsw.edu.au}

\author[D. Zhou]{Dejian Zhou$^*$}
\address{Dejian Zhou \\ School of Mathematics and Statistics,
Central South University, Changsha 410083, China}
\email{zhoudejian@csu.edu.cn}


\thanks{$*$ corresponding author.}

\begin{abstract}
We characterize noncommutative symmetric Banach spaces for which every bounded sequence admits either  a convergent subsequence, or a $2$-co-lacunary subsequence.
This extends the classical characterization, due to R\"abiger.
\end{abstract}

\maketitle

\section{Introduction}

Let $(X,\|\cdot\|_X)$ be a Banach space. A sequence $(x_n)_{n\geq1}$ in $X$ is said to be $2$-co-lacunary if there is a constant $C>0$ such that for any finite sequence $(\lambda_n)_{n\geq1}$ of complex numbers,
\begin{equation}\label{2 co-lacu}
\Big(\sum_{n\geq1} |\lambda_n|^2\Big)^{1/2}\leq C \Big\|\sum_{n\geq1}\lambda_n x_n\Big\|_X.
\end{equation}
The following remarkable result about $2$-co-lacunary subsequences was proved by Aldous and Fremlin \cite[Theorem 6]{AlFr1982}.
Let  $(\Omega,\mathcal{F},P)$  be a probability space. If $(x_n)_{n\geq1}$ is a bounded sequence in $L_1(\Omega)$, then either $(x_n)_{n\geq1}$ admits a convergent subsequence or a $2$-co-lacunary subsequence.
 R\"{a}biger \cite{Rabi1991} showed that for any Banach lattice $E$, the following properties are equivalent.
\begin{enumerate}
 \item Every bounded sequence in $E$ has a subsequence which is either convergent in norm, or is $2$-co-lacunary.
 \item Every semi-normalized disjoint sequence in $E$ has a subsequence which is $2$-co-lacunary.
\end{enumerate}

 A sequence $(x_n)_{n\ge 1}$ in $E$ is said to be semi-normalized if $\inf_{n\ge 1}\|x_n\|_E>0$ and $\sup_{n\ge 1}\|{x_n}\|_E<\infty$.

Here we are interested in the extension of such results to noncommutative symmetric spaces, such as those studied in \cite{Kalton2008}.
We now state our first main result below.

\begin{theorem}\label{2-main-result}
	Suppose that $E$ is an order continuous  symmetric Banach function space such that $ E\subset (L_1+L_2)(0,\infty).$ Let $\mathcal{M}$ be a hyperfinite and semifinite von Neumann algebra acting on a separable Hilbert space $H$, equipped with a faithful normal semifinite trace $\tau$, and let $E(\mathcal M)$ be the symmetric space associated to $(\mathcal M,\tau)$, and $E$.

Then the following assertions are equivalent:
\begin{enumerate}[{\rm(i)}]
\item  Every bounded sequence in $E(\mathcal{M})$ admits either a convergent subsequence or a $2$-co-lacunary subsequence.
\item  Every sequence of  pairwise orthogonal elements  in $E(\mathcal{M})$ either converges to zero or contains a $2$-co-lacunary subsequence.
 \end{enumerate}
\end{theorem}

A sequence of operators $(x_n)_{n\ge 1}\subset E(\mathcal{M})$  is said to be pairwise orthogonal if, for all $m\ne n,$ we have $x_nx_m=x_n^{\ast}x_m=0.$
In other words, our result extends R\"abiger's characterization to symmetric spaces associated with semifinite and hyperfinite von Neumann algebras. Moreover, if $E(\mathcal{M})=L_1(\mathcal{M}),$  then Theorem \ref{2-main-result} goes back to \cite[Corollary 3.7]{Rand2006}.
The proof of Theorem \ref{2-main-result}, given in Section \ref{sec-32},  heavily depends on the recent result \cite[Lemma~38]{JSZ}.

The next theorem is our second main result. See Definition \ref{def equi} for the notion  of $E$-equi-integrability.

\begin{theorem}\label{main theorem 1}
Let $E$ be an order continuous symmetric Banach function space such that $E\subset (L_1+L_2)(0,\infty).$ Suppose that $\mathcal{M}$ is a semifinite von Neumann algebra, and $(x_n)_{n\geq1}$ is a martingale difference sequence in $E(\mathcal{M})$ such that
\begin{enumerate}[{\rm (i)}]
  \item \label{con1} $A=\inf\{\|x_n\|_{E(\mathcal{M})}:\ n\geq1\}>0$;
  \item \label{con2} $(x_n)_{n\geq1}$ is $E$-equi-integrable in $E(\mathcal{M})$.
\end{enumerate}
Then the martingale difference sequence $(x_n)_{n\geq1}$ is  $2$-co-lacunary in $E(\mathcal{M})$.
\end{theorem}

It was proved by Akemann \cite[Theorem II.2]{Ak1967} (see also \cite[Page 149, Theorem 5.4]{Tak1979I}) that for $K\subset L_1(\mathcal{M})$, $K$ is  $L_1$-equi-integrable  if and only if $K$ is relatively weakly compact. Then, if $E(\mathcal{M})=L_1(\mathcal{M})$ in the above theorem, Theorem \ref{main theorem 1} is just \cite[Theorem 3.6]{Rand2006}.

In \cite{Rand2006}, to prove Theorem \ref{2-main-result} for $E=L_1$ (see Corollary 3.7 there),  Parcet and Randrianantoanina firstly proved \cite[Theorem 3.6]{Rand2006}.
At the end of \cite{Rand2006}, a simple proof of \cite[Theorem 3.6]{Rand2006} due to Pisier, was given for finite von Neumann algebras.

Pisier's argument mainly depends on the fact that the spaces $L_p(\mathcal{M})$, $0<p<1$,  are of cotype $2$,  and the application of the weak $(1,1)$ inequality  for martingale transformations, established in \cite{Rand2002}.
As mentioned by the authors of \cite[p.~251]{Rand2006}, it  was unknown that if one could extend Pisier's argument to semifinite von Neumann algebra $\mathcal{M}$.
Our method of proof for Theorem~\ref{2-main-result} is completely distinct from that of \cite[Theorem~3.6]{Rand2006}, and may be seen as a development of Pisier's ideas.

Preliminaries are given in Section \ref{sec-2}. Sections \ref{sec-32} and \ref{sec-31} are devoted to the proofs of Theorems \ref{2-main-result} and \ref{main theorem 1}, respectively.
At the end of Section \ref{sec-32}, we give some concluding remarks and an example demonstrating that Theorem \ref{2-main-result} holds true for a concrete Orlicz function space.

 Throughout this paper, we write $A\lesssim B$ if there is a constant $c>0$ such that $A\leq c B.$ We write $A\approx B$ if both $A \lesssim B $ and $B\lesssim A$ hold, possibly with different constants.

\section{Preliminaries}\label{sec-2}

\subsection{Noncommutative symmetric spaces}\label{sec-21}
Throughout the paper, $\mathcal M$ is  a  semifinite von Neumann algebra  equipped with a distinguished   faithful normal semifinite trace $\tau$.  Assume that $\mathcal{M}$ is acting on a Hilbert space $H$. A closed densely defined operator $x$ on $H$ is said to be  affiliated with $\mathcal{M}$ if $x$ commutes with the commutant $\mathcal{M}'$ of $\mathcal{M}$. If $a$ is self-adjoint and if $a=\int_{\mathbb{R}}s de_s^a$ is its spectral decomposition, then for any Borel subset $B\subseteq \mathbb{R}$, we denote by $\chi_B(a)$ the corresponding spectral projection $\int_{\mathbb{R}}\chi_B(s)de_s^a$. An operator $x$ affiliated with $\mathcal{M}$ is called $\tau$-measurable if there exists $s>0$ such that
$\tau(\chi_{(s,\infty)}(|x|))<\infty$ (see e.g. \cite[Definition 1.2]{Fa1986}).
Denote by $\mathcal{S}(\mathcal{M})$ the topological $*$-algebra of all $\tau$-measurable operators.

For $x\in \mathcal{S}(\mathcal{M})$, the generalized singular-valued function $\mu(x)$ is defined by
$$\mu_t(x)=\inf\{s>0:\tau\big(\chi_{(s,\infty)}(|x|)\big)\leq t\},\quad t>0.$$
The function $t\mapsto \mu_t(x)$ is decreasing and right-continuous. For more  detailed study of the generalized singular-value function, see for example \cite{Fa1986}. If $\mathcal M=L_\infty(0,\infty)$ is the abelian von Neumann algebra, then, for a measurable function $f$, the function  $\mu(f)$ is just the decreasing rearrangement of $f$ (see \cite[Page 39]{Be1988}).

Let $L_0(0,\infty)$  denote the set of all equivalence classes of Lebesgue measurable functions on $(0,\infty)$. A Banach (or quasi-Banach) function space $(E,\|\cdot\|_E)$ on the interval $(0,\infty)$ is called symmetric if, for every $g\in E$ and for every measurable function $f\in L_0(0,\infty)$ with $\mu(f)\leq \mu(g)$, we have $f\in E$ and $\|f\|_E\leq \|g\|_E$.
We say $E(0,\infty)$ is order continuous if $\|x_\beta\|\downarrow 0$ whenever $0\leq x_\beta \downarrow 0\subset E(0,\infty)$. $E$ is order continuous if and only if it is separable \cite{CS1994}.

Following \cite{Kalton2008}, for a given symmetric Banach (or quasi-Banach) function space $(E,\|\cdot\|_E)$ , we define the corresponding noncommutative space on $(\mathcal {M},\tau)$ by setting
$$E(\mathcal{M}):=\{x\in \mathcal{S}(\mathcal{M} ):\mu(x)\in E\}.$$
The associated quasi-norm is
\begin{equation}\label{def-sp}
\|x\|_{E(\mathcal M)}=\|\mu(x)\|_E.
\end{equation}
It is shown in \cite{Kalton2008} that if $E(0,\infty)$ is a symmetric Banach space, then $E(\mathcal{M})$ is Banach. This result is extended to quasi-Banach spaces in \cite{Su2014}.
That is it is established in \cite{Su2014} that  if $E(0,\infty)$ is a symmetric quasi-Banach space, then $E(\mathcal{M})$ is quasi-Banach.

 A useful quasi-Banach space is the weak space $L_{1,\infty}(0,\infty)$ defined by
$$L_{1,\infty}(0,\infty)=\{f\in L_0(0,\infty): \|f\|_{L_{1,\infty}}<\infty\},$$
where $\|f\|_{L_{1,\infty}}=\sup_{t>0}t\mu_t(f)$. Then $L_{1,\infty}(\mathcal{M})$ can be defined according to \eqref{def-sp}.

As usual, for $0< p,q\leq \infty$, $(L_p+L_q)(0,\infty)$ is the sum of the quasi-Banach spaces $L_p(0,\infty)$ and $L_q(0,\infty)$. Here, the quasi-norm is given by the formula
$$\|f\|_{L_p+L_q}=\inf\{\|g\|_p+\|g\|_q:\,f=g+h\}.$$
The space $(L_p\cap L_q)(0,\infty)$ is the intersection  of the quasi-Banach spaces $L_p(0,\infty)$ and $L_q(0,\infty)$. Here, the quasi-norm is given by the formula
$$\|f\|_{L_p\cap L_q}=\max\{\|f\|_p,\|f\|_q\}.$$
According to \cite[Theorem 4.1]{Holmstedt}, for $0< p<q<\infty$, we have
\begin{equation}\label{Hol-1}
\|f\|_{L_p+L_q}\approx \Big(\int_0^1\mu_t(f)^pdt\Big)^{1/p}+
\Big(\int_1^\infty\mu_t(f)^qdt\Big)^{1/q}.
\end{equation}

The spaces $(L_p+L_q)(\mathcal M)$ and $(L_p\cap L_q)(\mathcal M)$ are defined by \eqref{def-sp}.
For the case where $E(0,\infty)$ is a symmetric Banach function space, the inclusions
\begin{equation}\label{inclusions}
(L_1\cap L_\infty)(\mathcal M)\subset E(\mathcal M)\subset (L_1+L_\infty)(\mathcal M)
\end{equation}
hold with the continuous embeddings.

Let $x,y\in (L_{1}+L_\infty)(\mathcal{M})$.  The operator $y$ is said to be submajorized by $x$, denoted by $y \prec \prec x$, if for all $t\geq 0$,
\begin{equation}\label{submajor}
\int_0^t \mu_s(y)ds \leq \int_0^t  \mu_s(x) ds.
\end{equation}
This definition is taken from \cite[Definition 3.3.1]{LSZ2013}, and we also refer the reader to \cite[Chapter 3.3]{LSZ2013} for more information.
 We say that $E(\mathcal{M})$ is fully symmetric if  $y\in E(\mathcal{M})$ and $\mu(x)\prec \prec \mu(y)$ imply $x\in E(\mathcal{M})$ and $\|x\|_{E(\mathcal{M})}\leq \|y\|_{E(\mathcal{M})}$.
We say that $E(\mathcal{M})$ has the Fatou property if $(x_n)_{n\geq1}\subset E(\mathcal{M})$ and $x\in \mathcal{S}(\mathcal{M})$ such that $x_n\to x$ for the measure topology, then $x\in  E(\mathcal{M})$ and $\|x\|_{ E(\mathcal{M})}\leq \liminf_{n\to\infty}\|x_n\|_{ E(\mathcal{M})}$.

\begin{remark}\label{rem-fully}
If $E(0,\infty)$ is an order continuous symmetric Banach function space, then $E(0,\infty)$  is fully symmetric (see e.g. \cite{DPPS2011}, \cite[Page 112]{DDS2014}, or \cite{KPS1982}). Hence, $E(\mathcal{M})$ is fully symmetric (\cite[Page 245]{DoPa2014}).

If $E(0,\infty)$  has the  Fatou property, then $E(\mathcal{M})$ also has the  Fatou property (see \cite[Theorem 54(iii)]{DoPa2014}).
\end{remark}

We need the following definition   introduced by Randrianantoanina \cite{Rand2002JOT}. We also point out that related notions had been considered earlier in \cite{CS1994}, and were studied extensively in \cite{DoPaS2016}, \cite{Rand2003}, \cite{RX2003} and \cite{SX2003}.

\begin{definition}[{\cite[Definition 2.5]{Rand2002JOT} and \cite[Definition 3.3]{DDS2007}}]\label{def equi}
 Let $E(0,\infty)$ be a symmetric quasi-Banach function space  and $K$ be a bounded subset of $E(\mathcal{M})$. We will say that $K$ is $E$-equi-integrable if for every decreasing sequence $(e_n)_{n\geq1}$ of projections with $e_n\downarrow 0$,
$$\lim_{n\to \infty }\sup_{x\in K}\|e_nxe_n\|_{E(\mathcal{M})}=0.$$
\end{definition}

The following lemma is taken from \cite{DoPaS2016} (see also \cite{DDS2007} and \cite{Rand2002JOT}).

\begin{lemma}[{\cite[Theorem 3.4]{DoPaS2016}}] \label{e imply com}
Let $E(0,\infty)$ be an order continuous symmetric Banach function space.
If $(x_n)_{n\geq1} \subset E(\mathcal{M})$ is bounded and $E$-equi-integrable, then
$\|x_n\|_{E(\mathcal{M})}\to 0$ if and only if $x_n\to 0$ in measure topology.
\end{lemma}

\subsection{Hyperfinite von Neumann algebras}\label{sec-22}
A von Neumann algebra is called hyperfinite if coincides with the weak closure of a increasing net of finite dimensional subalgebras (see e.g. \cite{Connes1976} or \cite[Page 49]{SiSm2008}).

Consider $\mathcal{M}$ a hyperfinite and infinite von Neumann algebra acting on a separable Hilbert space $H$. Denote by $\mathcal{R}$ the hyperfinite $\mathrm{II}_1$ factor (see for example \cite{Connes1976}).  Then $\mathcal{M}$ is trace preserving $*$-isomorphic to a  von Neumann subalgebra of $\mathcal{R}\bar{\otimes} B(H)$.
Indeed, we have that (see e.g. \cite[Theorem V.1.19]{Tak1979I})
$$\mathcal{M}=\mathcal{M}_{\mathrm{I}}\oplus \mathcal{M}_{\mathrm{II}_1}\oplus \mathcal{M}_{\mathrm{II}_{\infty}}.$$
By applying \cite[Proposition 6.5]{Connes1976}, we have that  (see also \cite[Page 59]{HRS2003}) every finite hyperfinite von Neumann algebra $\mathcal{N}$ is ($*$-isomorphic to) a countable direct sum of von Neumann algebras of the form $\mathcal{A}\bar{\otimes} \mathcal{B}$, where $\mathcal{A}$ is abelian and $\mathcal{B}$ is either $B(\ell_2^n)$ for some $n<\infty$ or $\mathcal{R}$. Note that $\mathcal{A}$ can be realized as a subalgebra of $L_{\infty}(0,\infty)=L_{\infty}(0,1)\bar{\otimes}\ell_{\infty}.$ Hence, $\mathcal{A}$ can be realised as a subalgebra in $\mathcal{R}\bar{\otimes}B(H).$ Obviously, $\mathcal{B}$ is a subalgebra in $B(H).$ Thus, $\mathcal{A}\bar{\otimes}\mathcal{B}$ can be realised as a subalgebra in $\mathcal{R}\bar{\otimes}B(H).$ Then $\mathcal{N}$ is a subalgebra of $\mathcal{R}\bar{\otimes}B(H)\bar{\otimes}\ell_{\infty},$ which is trace preserving  $*$-isomorphic to a subalgebra of $\mathcal{R}\bar{\otimes} B(H)\bar{\otimes} B(H)=\mathcal{R}\bar{\otimes} B(H)$.

According to \cite{Connes1976} (see also \cite[Page 60]{HRS2003}), $\mathcal{M}_{\mathrm{II}_{\infty}}$ is $\mathcal{N}\bar{\otimes} B(H)$ where $\mathcal{N}$ is a
 finite hyperfinite von Neumann algebra. Hence, $\mathcal{M}_{\mathrm{II}_{\infty}}$  is trace preserving $*$-isomorphic to a von Neumann subalgebra of  $\mathcal{R}\bar{\otimes} B(H)\bar{\otimes} B(H)=\mathcal{R}\bar{\otimes} B(H)$.

If $\mathcal{M}_{\mathrm{I}}$ is infinite, then $\mathcal{M}_{\mathrm{I}}$ is $*$-isomorphic to a countable direct sum of  von Neumann algebras $\mathcal{A}\bar{\otimes} \mathcal{B}$, where $\mathcal{A}$ is abelian and $\mathcal{B}$ is $B(\ell_2^n)$ for some $n<\infty$ or $B(H)$ (\cite[Theorem V.1.27]{Tak1979I}),  and consequently,
 $\mathcal{M}_{\mathrm{I}}$  is trace preserving $*$-isomorphic to a von Neumann subalgebra of $\mathcal{R}\bar{\otimes}\ell_\infty \bar{\otimes}B(H)$, which is a subalgebra of $\mathcal{R}\bar{\otimes} B(H)$.

\subsection{Noncommutative martingale differences}
In this subsection,  we  review the basics of noncommutative martingales.
Let
$(\mathcal{M}_n)_{n\geq1}$ be an increasing sequence of von Neumann
subalgebras of $\mathcal{M}$ such that the union of the $\mathcal{M}_n$'s is
weakly dense in $\mathcal{M}$. Assume that for every $n\geq 1$, there exists a normal $\tau$-invariant   conditional
expectation from  $\mathcal{M}$ onto  $\mathcal{M}_n$. In fact,  for the case where $\mathcal{M}$ is finite then such  conditional expectations always exist (see \cite[Lemma 3.6.2]{SiSm2008} or \cite{Tak1979I}). If the  restriction of $\tau$ on $\mathcal{M}_n$ remains semifinite, then such  conditional expectations  exist (\cite[Page 332]{Tak1979I}).  Since $\mathcal{E}_n$ is $\tau$-invariant, it extends to a contractive projection from $L_p(\mathcal{M},\tau)$ onto $L_p(\mathcal{M}_n,\tau_n)$ for all $1\leq p\leq \infty$, where $\tau_n$ denotes the restriction of $\tau$ on $\mathcal{M}_n$.

\begin{definition}
A sequence $x=(x_k)_{k\geq1} $ in $(L_1+L_{\infty})(\mathcal{M})$ is said to be a sequence of martingale differences if $x_k\in \mathcal{M}_k$ for each $k\geq1$ and $\mathcal{E}_{k-1}(x_k)=0$ for every $k\geq2$.
\end{definition}

The following  important result is a  corollary of \cite[Theorem 3.1]{Rand2002}.
\begin{lemma}\label{CR lemma}
Let $(x_k)_{k\geq1}\subset L_1(\mathcal{M})$ be a sequence of martingale differences. Then
$$\Big\|\sum_{k\geq1}r_k\otimes x_k\Big\|_{L_{1,\infty}(L_\infty(0,1)\bar{\otimes} \mathcal{M})}\leq c_{{\rm abs}}\sup_n\Big\|\sum_{k=1}^n x_k\Big\|_1,$$
where $(r_k)_{k\ge 1}$ is the sequence of Rademacher functions (see, for example, \cite{LiTz1979}) and $L_{\infty}(0,1)\bar{\otimes} \mathcal{M}$ is the tensor product von Neumann algebra (see for example \cite[Page 183]{Tak1979I}).
\end{lemma}

\begin{proof}
Since $(x_k)_{k\geq1}$ is a martingale difference sequence in $L_1(\mathcal{M})$, so is $(\chi_{(0,1)}\otimes x_k)_{k\geq 1}$ in $L_1(L_{\infty}(0,1)\otimes \mathcal{M})$ with respect to the filtration $(L_{\infty}(0,1)\otimes \mathcal{M}_k)_{k\geq1}$. Then,
it follows from \cite[Theorem 3.1]{Rand2002} that
\begin{align*}
\Big\|\sum_{k\geq1}r_k\otimes x_k\Big\|_{L_{1,\infty}(L_\infty(0,1)\bar{\otimes} \mathcal{M})}&=\Big\|\sum_{k\geq1}(r_k\otimes 1)(\chi_{(0,1)}\otimes x_k)\Big\|_{L_{1,\infty}(L_\infty(0,1)\bar{\otimes} \mathcal{M})}\\
&\leq c_{{\rm abs}} \sup_n \Big\| \sum_{k=1}^n \chi_{(0,1)}\otimes x_k \Big\|_{L_1(L_\infty(0,1)\bar{\otimes} \mathcal{M})}\\
&=c_{{\rm abs}} \sup_n \Big\| \sum_{k=1}^n  x_k \Big\|_{1}.
\end{align*}
\end{proof}

\section{Proof of Theorem \ref{2-main-result}}\label{sec-32}

In this section, we prove  Theorem \ref{2-main-result}.
We will need the following perturbation lemma from \cite[Lemma 2]{AlFr1982}.

\begin{lemma}\label{perturbation}
Let $X$ be a normed space,  and let $(x_n)_{n\geq1}$ be a bounded sequence in $X$. Then
\begin{enumerate}[{\rm (i)}]
\item \label{p-lem-1} if $(x_n)_{n\geq1}$ is $2$-co-lacunary and $x\in X$, then there exists $m\in \mathbb{N}$ such that $(x_n-x)_{n\geq m}$ is $2$-co-lacunary;
\item \label{p-lem-2} if $(x_n)_{n\geq1}$ is $2$-co-lacunary and $(y_n)_{n\geq1}\subset X$ with $\sum_{n\geq1}\|x_n-y_n\|_X$ being convergent, then there exists $m\in \mathbb{N}$ such that $(y_n)_{n\geq m}$ is $2$-co-lacunary.
\end{enumerate}
\end{lemma}

The following assertion can be found in \cite{Rand2006}.

\begin{theorem}[{\cite[Corollary 3.7]{Rand2006}}]\label{L1 theorem} Let $\mathcal{M}$ be hyperfinite and semifinite. Every bounded sequence $(x_n)_{n\geq1}$ in $L_1(\mathcal{M})$ admits either a convergent subsequence or a $2$-co-lacunary subsequence.
\end{theorem}

\begin{corollary}\label{l1+l2 thm} Let $H$ be a separable Hilbert space and $\mathcal{R}$ be the hyperfinite $\mathrm{II}_1$ factor.  Every bounded sequence $(x_n)_{n\geq1}$ in $(L_1+L_2)(\mathcal{R}\bar{\otimes}B(H))$ admits either a convergent subsequence or a $2$-co-lacunary subsequence.
\end{corollary}
\begin{proof}

 It was proved in \cite[Lemma 38]{JSZ} that there exists an isomorphic embedding $T:(L_1+L_2)(\mathcal{R}\bar{\otimes}B(H))\to L_1(\mathcal{R}).$
Obviously, the sequence $(x_n)_{n\geq1}$ (in $(L_1+L_2)(\mathcal{M})$) is  2-co-lacunary (respectively, convergent) if and only if the sequence $(Tx_n)_{n\geq1}$ (in $L_1(\mathcal{R})$) is 2-co-lacunary (respectively, convergent). Now the assertion follows from Theorem \ref{L1 theorem}.
\end{proof}

\begin{corollary}\label{l1+l2 cor} Let $\mathcal{M}$ be a hyperfinite and semifinite von Neumann algebra acting on a separable Hilbert space $H$. Every bounded sequence $(x_n)_{n\geq1}$ in $(L_1+L_2)(\mathcal{M})$ admits either a convergent subsequence or a $2$-co-lacunary subsequence.
\end{corollary}

\begin{proof}
If $\mathcal{M}$ is finite, then the corollary is just Theorem \ref{L1 theorem}. We now consider the case when $\mathcal{M}$ is infinite.  Note that $\mathcal{M}$ is trace preserving $*$-isomorphic to a subalgebra of $\mathcal{R}\bar{\otimes}B(H)$ (see Subsection \ref{sec-22}).
Then $(L_1+L_2)(\mathcal{M})$ is isomorphic to a subspace of $(L_1+L_2)(\mathcal{R}\bar{\otimes}B(H))$, and the assertion of the corollary follows from Corollary \ref{l1+l2 thm}.
\end{proof}

We need the following result from \cite{CDS1997}.

\begin{lemma}[{\cite[Theorem 2.5]{CDS1997}}]\label{measure lem}
Let  $(\mathcal{M},\tau)$ be semifinite and $E(0,\infty)$ be an order continuous symmetric Banach function space. If $(x_n)_{n\geq1}\subset E(\mathcal{M})$ is a sequence of elements which converges to zero in the measure topology, then there exists
a subsequence $(x_{n_k})_{k\geq1}$ and two sequences of mutually orthogonal projections $(p_k)_{k\geq1}$ and $(q_k)_{k\geq1}$ in $\mathcal{M}$ such that
$$\|x_{n_k}-p_kx_{n_k}q_k\|_{E(\mathcal{M})}\to 0.$$
\end{lemma}

\begin{lemma}\label{to0 in measure} Suppose that $E$ satisfies the Assumption (ii) in Theorem \ref{2-main-result}. Every bounded sequence $(x_n)_{n\geq1}\subset E(\mathcal{M})$ which converges to zero in the measure topology admits either a convergent subsequence or a $2$-co-lacunary subsequence.
\end{lemma}
\begin{proof} If the sequence $(x_n)_{n\geq1}$ converges to $0$ in $E(\mathcal{M}),$ then the assertion follows. Otherwise, passing to a subsequence if needed, we may assume that the sequence $(x_n)_{n\geq1}$ is semi-normalised.

By Lemma \ref{measure lem}, there are a subsequence $(x_{n_k})_{k\geq1}$ and sequences $(p_k)_{k\geq1}$ and $(q_k)_{k\geq1}$ of mutually orthogonal projections in $\mathcal{M}$ such that
\begin{equation}\label{xy1}
\|x_{n_k}-p_kx_{n_k}q_k\|_{E(\mathcal{M})}\to 0.
\end{equation}
Denote for brevity $u_k=p_kx_{n_k}q_k.$ According to \eqref{xy1}, passing to a further subsequence if it is necessary, we may assume
$$\sum_{l\geq1}\|x_{n_k}-u_k\|_{E(\mathcal{M})}<\infty.$$

Since the sequence $(x_n)_{n\geq1}$ is semi-normalised, $(u_k)_{k\geq1}$ is also semi-normalized. So is the sequence $(|u_k|)_{k\geq1}.$ The latter sequence consists of pairwise orthogonal elements. By Assumption (ii) of Theorem \ref{2-main-result}, there exists a 2-co-lacunary subsequence $(|u_{k_l}|)_{l\geq1}.$ Observe that for any $(\lambda_l)_{l\geq1}\subset \mathbb{C}$,
\begin{align*}
|\sum_l\lambda_l u_{k_l}|^2&=\sum_{l_1,l_2}\overline{\lambda_{l_1}}\lambda_{l_2} q_{k_{l_1}}x_{n_{k_{l_1}}}^{\ast}p_{k_{l_1}}p_{k_{l_2}}x_{n_{k_{l_2}}}q_{k_{l_2}}\\
&=\sum_l|\lambda_l|^2 q_{k_l}x_{n_{k_l}}^{\ast}p_{k_l}x_{n_{k_l}}q_{k_l}=\sum_l|\lambda_l|^2|u_{k_l}|^2=|\sum_l\lambda_l|u_{k_l}||^2.
\end{align*}
Since $(|u_{k_l}|)_{l\geq1}$ is 2-co-lacunary, it follows that
$$\Big(\sum_{l\geq1}|\lambda_l|^2\Big)^{1/2}\leq c\Big\|\sum_l\lambda_l|u_{k_l}|\Big\|_{E(\mathcal{M})}=c\Big\|\sum_l\lambda_lu_{k_l}\Big\|_{E(\mathcal{M})}.$$
This means the sequence $(u_{k_l})_{l\geq1}$ is 2-co-lacunary.

By Lemma \ref{perturbation}\eqref{p-lem-2}, the subsequence $(x_{n_{k_l}})_{l\geq1}$ is 2-co-lacunary. Thus $(x_n)_{n\geq1}$  contains a 2-co-lacunary subsequence.
\end{proof}

\begin{lemma}\label{Fatou corollary} Let $E(0,\infty)$ be an order continuous  symmetric Banach function space. If every sequence of pairwise orthogonal elements in $E(0,\infty)$ admits either a convergent subsequence or a $2$-co-lacunary subsequence, then $E(0,\infty)$ has the Fatou property.
\end{lemma}

\begin{proof}
By the assumption, no sequence of disjointly supported elements in $E$ spans $c_0.$ By \cite[Theorem 6.5(v)$\Rightarrow $(i)]{DoPaS2016}, $E$ has the Fatou property.
\end{proof}

We are now ready to prove our main result Theorem \ref{2-main-result}.

Recall that $E$ is an order continuous symmetric Banach function space such that $ E\subset (L_1+L_2)(0,\infty).$ Assume that $\mathcal{M}$ is semifinite and hyperfinite.

\begin{proof}[Proof of Theorem \ref{2-main-result}] The implication ${\rm (i)}\Rightarrow {\rm (ii)}$ is easy. Indeed, let  $(x_n)_{n\geq1}$ in $E(\mathcal{M})$ be a bounded sequences, and both right and left disjointly supported. If $(x_n)_{n\geq1}$ is convergent in $E(\mathcal{M}),$ then it converges to zero. However, $(x_n)_{n\geq1}$ is semi-normalized; hence, it does not have a subsequence which converges to $0.$ Thus, (ii) follows from  (i).

We now concentrate on the implication ${\rm (ii)}\Rightarrow {\rm (i)}.$ It suffices to consider a sequence $(x_n)_{n\geq1}$ in $E(\mathcal{M})$ which is semi-normalised.

\textbf{Case A}: Suppose that the sequence $(x_n)_{n\geq1}$ converges to zero in $(L_1+L_2)(\mathcal{M}).$

In this case, the sequence $(x_n)_{n\geq1}$ converges to zero in measure, and hence the application of Lemma \ref{to0 in measure} yields the assertion of Theorem \ref{2-main-result}.

\textbf{Case B}: Suppose that the sequence $(x_n)_{n\geq1}$ does not converge to zero in $(L_1+L_2)(\mathcal{M}).$

Choose $\delta>0$ and a subsequence $x_{n_k}$ such that
$$\lim_{k\to\infty}\|x_{n_k}\|_{L_1+L_2}=\delta.$$
By Corollary \ref{l1+l2 cor}, the sequence $(x_{n_k})_{k\geq1}$ either contains a subsequence $(x_{n_{k_l}})$ which is either 2-co-lacunary in $(L_1+L_2)(\mathcal{M})$ or converges in $(L_1+L_2)(\mathcal{M}).$

If $(x_{n_{k_l}})$ is 2-co-lacunary in $(L_1+L_2)(\mathcal{M}),$ then for any complex sequence $(\lambda_l)_l$,
$$\Big\|\sum_{l\geq1}\lambda_lx_{n_{k_l}}\Big\|_E\geq \Big\|\sum_{l\geq1}\lambda_lx_{n_{k_l}}\Big\|_{L_1+L_2}\gtrsim  \Big(\sum_{l\geq1} |\lambda_{l}|^2\Big)^{1/2}.$$
Hence, the sequence $(x_{n_{k_l}})$ is 2-co-lacunary in $E(\mathcal{M}).$

If $(x_{n_{k_l}})$ is convergent in $(L_1+L_2)(\mathcal{M}),$ then we denote the limit by $x.$ By Assumption (ii) of Theorem \ref{2-main-result} and Lemma \ref{Fatou corollary}, $E$ has the Fatou property and, hence, so does $E(\mathcal{M})$ (see Remark \ref{rem-fully}). By the Fatou property, $x\in E(\mathcal{M}).$ Set $y_l=x_{n_{k_l}}-x\in E(\mathcal{M}).$ Note that $(y_l)_{l\geq1}$ converges to $0$ in measure. If $y_l\to0$ in $E(\mathcal{M}),$ then $x_{n_{k_l}}\to x$ in $E(\mathcal{M}).$

If $(y_l)_{l\geq1}$ does not converge to $0$ in $E(\mathcal{M}),$ then, by Lemma \ref{to0 in measure}, we construct a 2-co-lacunary subsequence of the sequence $(y_l)_{l\geq1}.$ By Lemma \ref{perturbation}\eqref{p-lem-1}, there is a 2-co-lacunary subsequence of the sequence $(x_{n_{k_l}})_{l\geq1}.$
 This completes the proof in Case B.
\end{proof}

At the end of this section,  we present some concluding remarks for Theorem \ref{2-main-result}.
 First of all, we demonstrate that the assumption $E\subset (L_1+L_2)(0,\infty)$ is necessary for the validity of Theorem \ref{2-main-result}.

 Let $E(0,\infty)$ be a symmetric Banach function space. If every bounded sequence $(f_n)_{n\geq1}\subset E(0,\infty)$ admits either a convergent subsequence or a 2-co-lacunary subsequence, then $E\subset (L_1+L_2)(0,\infty)$. Indeed, take $f\in E(0,\infty)$.  We have
\begin{align*}
\|f\|_{E}=\|\mu(f)\|_E\geq \Big\| \sum_{k\geq0} \mu_{k+1}(f) \chi_{[k,k+1)}\Big\|_{E}.
\end{align*}
Since the sequence $(\chi_{[k,k+1)})_{k\geq0}$  is pairwise disjoint and semi-normalized, it does not have a convergent subsequence. Hence, by the assumption, it is 2-co-lacunary. Then we have
\begin{align*}
\|f\|_{E}&\gtrsim \Big(\sum_{k\geq0} \mu_{k+1}(f)^2\Big)^{1/2}\geq \Big(\sum_{k\geq0} \int_{k+1}^{k+2}\mu_s(f)^2ds\Big)^{1/2}=\Big(\int_1^\infty \mu_s(f)^2ds\Big)^{1/2}
\end{align*}
Observe that $\|\mu(f)\chi_{(0,1)}\|_1\overset{\eqref{inclusions}}\leq \|\mu(f)\chi_{(0,1)}\|_E \leq \|\mu(f)\|_E=\|f\|_E$. Hence, by \eqref{Hol-1}, we have
$$\|f\|_{L_1+L_2}\approx \int_0^1\mu_s(f)\chi_{(0,1)} ds+ \Big(\int_1^\infty \mu_s(f)^2ds\Big)^{1/2}
\lesssim \|f\|_E,$$
which means $E\subset (L_1+L_2)(0,\infty)$.

We give an example to show that there exists an order continuous symmetric function space   $E\subset(L_1+L_2)(0,\infty)$ cannot contain a $2$-co-lacunary subsequence.
Let us assume that $E=M_\psi^0$ is a ``separable part"  of Marcinkiewicz space $M_\psi$, where $\psi$ is continuous concave function on $[0,\infty)$ such that $\psi(0)=0$ (see e.g. \cite{KPS1982} and \cite{AS2007}) such that $M_\psi\subset (L_1+L_2)(0,\infty)$.
By \cite[Proposition 2.1]{AS2007}, we know that any normalized  disjointly supported sequence in $E$ for which every member is a scalar multiple of a characteristic function of a measurable subset of $(0,1)$ contains a subsequence equivalent a standard vector basis in $c_0$. Hence, such a subsequence cannot contain a $2$-co-lacunary subsequence.

Recall that a Banach lattice $X$ is said to satisfy a lower $p$-estimate ($1<p<\infty$)  (see, e.g., \cite[Definition 1.f.4]{LiTz1979}) if there is a constant $M$ such that for every choice of pairwise disjoint elements $(x_i)_{i=1}^n \subset X$, we have
$$\Big(\sum_{i=1}^n\|x_i\|_X^p\Big)^{1/p}\leq M \Big\|\sum_{i=1}^nx_i\Big\|_X.$$
Now consider the von Neumann algebra $B(H)$, where $H$ is a separable Hilbert space, and the standard trace $\operatorname{tr}$. Our symmetric spaces $E(B(H),\operatorname{tr})$ associated to the algebra become ideals of $B(H)$.
For every symmetric sequence space $E$ satisfying lower $2$-estimate, the ideal $C_E=E(B(H))$ is contained between the Schatten classes $C_1$ and $C_2$.
 See, for example, \cite[Part~II]{LSZ2013} for extensive discussion of the correspondence between sequence spaces and ideals of $B(H)$.
We note that by \cite[Proposition~2.3]{Su1996}, we have that for a symmetric sequence space $E$, any pairwise disjoint sequence in $C_E$ is isometrically isomorphic to disjoint basic sequence in $E$. Theorem~\ref{2-main-result} may now be rephrased in terms of ideals.

\begin{corollary} Let symmetric sequence space $E$ be separable. If $E$ satisfies lower $2$-estimate, then every bounded sequence $(x_n)_{n\geq1}$ in the ideal $C_E$ admits either a convergent subsequence, or  a $2$-co-lacunary subsequence.
\end{corollary}

\section{Proof of Theorem \ref{main theorem 1}}\label{sec-31}

In this section, we prove  Theorem \ref{main theorem 1}, which builds upon Pisier's argument in \cite[Page 251]{Rand2006}.

Let $0<q<\infty$. A quasi-Banach space $(X,\|\cdot\|_X)$ is $q$-concave (see e.g. \cite[Definition 1.d.3]{LiTz1979}) if for every finite sequence $(x_k)_{k=1}^n\subset X$ there exists a constant $K_q$ such that
\begin{equation*}\label{cotype}
\Big(\sum_{i=1}^n\|x_i\|_X^q\Big)^{1/q}\leq K_q\Big\|\Big(\sum_{i=1}^n|x_i|^q\Big)^{1/q}\Big\|_X.
\end{equation*}
We need to show that for each $0<p<1$, $(L_p+L_2)(\mathcal{M})$ is $2$-concave. It was shown in \cite[Theorem 3.8]{DDS2014} that if a symmetric Banach space $E(0,\infty)$ is $q$-concave, then  $E(\mathcal{M})$ is $q$-concave. However, we may not apply \cite[Theorem 3.8]{DDS2014} here, since $(L_p+L_2)(0,\infty)$ is  quasi-Banach.

\begin{lemma}\label{cor-2-concave}
Suppose that $0<p\leq 2$ and $\mathcal{M}$ is a semifinite von Neumann algebra. Then
$(L_p+L_2)(\mathcal{M})$ is $2$-concave, i.e., that is there exists a constant $c_p$ such that for all $n\geq1$ and $(x_k)_{k=1}^n\subset (L_p+L_2)(\mathcal{M})$, we have
$$\Big(\sum_{k=1}^n\|x_k\|_{(L_p+L_2)(\mathcal{M})}^2\Big)^{1/2}\leq c_p\Big\|\Big(\sum_{k=1}^n|x_k|^2\Big)^{1/2}\Big\|_{(L_p+L_2)(\mathcal{M})}.$$
\end{lemma}
\begin{proof}
 Using the fact that $L_w(\mathcal{M})$ is cotype 2 for any $0<w\leq 2$ (see e.g. \cite[Corollary 5.5]{PX2003}), \cite[Theorem 1.e.13]{LiTz1979} and noncommutative Khintchine (see e.g. \cite[Remark 3.3]{Cai2018}), we have
\begin{align}\label{refree1}
\Big(\sum_{k=1}^n\|x_k\|_{w}^2\Big)^{1/2}&\leq  c_w \int_0^1 \Big\|\sum_{k=1}^nr_k(t)x_k\Big\|_{w}dt \nonumber \\
&\leq c_w \Big(\int_0^1 \Big\|\sum_{k=1}^nr_k(t)x_k\Big\|_{w}^wdt\Big)^{1/w} \nonumber \\
&\leq c_w \Big\|\Big(\sum_{k=1}^n|x_k|^2\Big)^{1/2}\Big\|_{w},
\end{align}
where $\{r_k\}_{k\ge 1}$ is the sequence of Rademacher functions.

On the other hand, it is well known that for any symmetric quasi-Banach space $E$, the map $J: E(\mathcal{M},\ell_2^c)\rightarrow E(\mathcal{M}\bar{\otimes}B(\ell_2))$ defined by setting $J((a_k)_{k\geq1})=\sum_{k\geq1}a_k\otimes e_{k,1}$ is an isometry.

It now follows from \eqref{refree1} that for every $0<w\leq 2$, the map $\Theta: L_w(\mathcal{M}\bar{\otimes}B(\ell_2))\rightarrow\ell_2(L_w(\mathcal{M}))$ defined by
$$\Theta((a_{i,j})_{i,j\geq1})=(a_{i,1})_{i\geq1}$$
is bounded. In particular, we have that $\Theta$ is simultaneously bounded as maps $\Theta: L_p(\mathcal{M}\bar{\otimes}B(\ell_2))\rightarrow\ell_2((L_p+L_2)(\mathcal{M}))$
and $\Theta: L_2(\mathcal{M}\bar{\otimes}B(\ell_2))\rightarrow\ell_2((L_p+L_2)(\mathcal{M}))$.
This implies that
$$\Theta: (L_p+L_2)(\mathcal{M}\bar{\otimes}B(\ell_2))\rightarrow \ell_2((L_p+L_2)(\mathcal{M}))$$
is bounded. Composing this with the map $J$, we conclude
$$\Theta\circ J: (L_p+L_2)(\mathcal{M}\bar{\otimes}B(\ell_2))\rightarrow \ell_2((L_p+L_2)(\mathcal{M}))$$
is bounded. This actually means $(L_p+L_2)(\mathcal{M})$ is 2-concave.
\end{proof}

We shall apply the noncommutative Khintchine inequality established in \cite{Cai2018}.

\begin{lemma}\label{cotype lemma}
Suppose that $0<p\leq 2$ and $\mathcal{M}$ is a semifinite von Neumann algebra.
For every finite sequence $(x_k)\subset (L_p+L_2)(\mathcal{M})$, we have
$$\Big\|\sum_{k\geq1}r_k\otimes x_k\Big\|_{(L_p+L_2)(L_\infty(0,1)\bar{\otimes} \mathcal{M})}\geq c_p\Big(\sum_{k\geq1}\|x_k\|_{(L_p+L_2)(\mathcal{M})}^2\Big)^{\frac12},$$
where $\{r_k\}_{k\ge 1}$ is the sequence of Rademacher functions.
\end{lemma}
\begin{proof} According to \cite[Remark 3.3]{Cai2018},  the noncommutative Khintchine inequality holds in the space $L_p+L_2.$ That is
\begin{align*}
\Big\|\sum_{k\geq1}r_k&\otimes x_k\Big\|_{(L_p+L_2)(L_\infty(0,1)\bar{\otimes} \mathcal{M})}\approx \\ &\inf_{x_k=y_k+z_k}\Big\{ \Big\|\Big(\sum_{k\geq1} |y_k|^2\Big)^{1/2}\Big\|_{(L_p+L_2)(\mathcal{M})}+\Big\|\Big(\sum_{k\geq1} |z_k^*|^2\Big)^{1/2}\Big\|_{(L_p+L_2)( \mathcal{M})}\Big\}.
\end{align*}
By Lemma \ref{cor-2-concave}, we have that
$$\Big(\sum_{k\geq 1}\|y_k\|_{(L_p+L_2)(\mathcal{M})}^2\Big)^{1/2}\leq c_p\Big\|\Big(\sum_{k\geq 1}|y_k|^2\Big)^{1/2}\Big\|_{(L_p+L_2)(\mathcal{M})}$$
and
$$\Big(\sum_{k\geq 1}\|z_k\|_{(L_p+L_2)(\mathcal{M})}^2\Big)^{1/2}\leq c_p\Big\|\Big(\sum_{k\geq 1}|z_k^*|^2\Big)^{1/2}\Big\|_{(L_p+L_2)(\mathcal{M})}.$$
Thus
\begin{align*}
\Big\|\sum_{k\geq1}r_k\otimes x_k&\Big\|_{(L_p+L_2)(L_\infty(0,1)\bar{\otimes} \mathcal{M})}\gtrsim \\ &\inf_{x_k=y_k+z_k}\Big\{ \Big(\sum_{k\geq 1}\|y_k\|_{(L_p+L_2)(\mathcal{M})}^2\Big)^{1/2}+\Big(\sum_{k\geq 1}\|z_k\|_{(L_p+L_2)(\mathcal{M})}^2\Big)^{1/2}\Big\}.
\end{align*}
By the quasi-triangle inequality, we have that
$$\Big(\sum_{k\geq1}\|y_k\|_{(L_p+L_2)(\mathcal{M})}^2\Big)^{1/2}+\Big(\sum_{k\geq1}\|z_k\|_{(L_p+L_2)(\mathcal{M})}^2\Big)^{1/2}\gtrsim \Big(\sum_{k\geq1}\|x_k\|_{(L_p+L_2)(\mathcal{M})}^2\Big)^{1/2}.$$
The assertion follows by combining the last two inequalities.
\end{proof}

\begin{lemma}\label{martingale lemma}
 Let $\mathcal{M}$ be a semifinite von Neumann algebra. Assume that
$(\mathcal{M}_n)_{n\geq1}$ is an increasing sequence of von Neumann
subalgebras of $\mathcal{M}$ such that the union of the $\mathcal{M}_n$'s is
$w^*$-dense in $\mathcal{M}$ and $(\mathcal{E}_n)_{n\geq1}$ are the corresponding $\tau$-invariant conditional expectations.
Let $x\in (L_1+L_2)(\mathcal{M})$ and $x_k=\mathcal{E}_k(x)-\mathcal{E}_{k-1}(x)$, $k\geq1$. Then, there is a positive constant $c_p$ depends on $p$ such that
$$\Big(\sum_{k\geq1}\|x_k\|_{(L_p+L_2)(\mathcal{M})}^2\Big)^{\frac12}\leq c_p\|x\|_{(L_1+L_2)(\mathcal{M})},\quad 0<p<1.$$
\end{lemma}
\begin{proof}
Since $x\in (L_1+L_2)(\mathcal{M})$, there exist $y\in L_1(\mathcal{M})$ and $z\in L_2(\mathcal{M})$ such that $x=y+z$, $\|y\|_1\leq 2\|x\|_{(L_1+L_2)(\mathcal{M})}$ and $\|z\|_2\leq 2\|x\|_{(L_1+L_2)(\mathcal{M})}.$
For every $k\geq1$, set $y_k=\mathcal{E}_k(y)-\mathcal{E}_{k-1}(y)$ and $z_k=\mathcal{E}_k(z)-\mathcal{E}_{k-1}(z).$ Now, $x_k=y_k+z_k$. Since $L_p+L_2$ is a quasi-space for $0<p<1$, there is a constant $c_p>0$ such that
$$\|x_k\|_{L_p+L_2}\leq c_p(\|y_k\|_{L_p+L_2}+\|z_k\|_{L_p+L_2})\leq c_p(\|y_k\|_{L_p+L_2}+\|z_k\|_2).$$
Thus
$$c_p^{-1}\Big(\sum_{k\geq1}\|x_k\|_{L_p+L_2}^2\Big)^{\frac12}\leq \Big(\sum_{k\geq1}\|y_k\|_{L_p+L_2}^2\Big)^{\frac12}+\Big(\sum_{k\geq1}\|z_k\|_2^2\Big)^{\frac12}.$$
Observing that $0<p<1$, we infer that $L_{1,\infty}(0,\infty)\subset (L_p+L_2)(0,\infty)$. Then, by Lemma \ref{cotype lemma} and Lemma \ref{CR lemma}, we have
\begin{align*}
\Big(\sum_{k\geq1}\|y_k\|_{(L_p+L_2)(\mathcal{M})}^2\Big)^{\frac12}&\leq c_p\Big\|\sum_{k\geq1}r_k\otimes y_k\Big\|_{(L_p+L_2)(L_\infty(0,1)\bar{\otimes} \mathcal{M})}\\
&\leq c_pc_p'\Big\|\sum_{k\geq1}r_k\otimes y_k\Big\|_{L_{1,\infty}(L_\infty(0,1)\bar{\otimes} \mathcal{M})}\leq c_pc_p'\|y\|_1.
\end{align*}
Combining the above estimates, we obtain
$$c_p^{-1}\Big(\sum_{k\geq1}\|x_k\|_{(L_p+L_2)(\mathcal{M})}^2\Big)^{\frac12}\leq c_pc_p'\|y\|_1+\|z\|_2\leq 2(1+c_pc_p')\|x\|_{(L_1+L_2)(\mathcal{M})}.$$
The proof is complete.
\end{proof}

Now we are in a position to prove Theorem \ref{main theorem 1}.

\begin{proof}[Proof of Theorem \ref{main theorem 1}]
Fix $0<p<1$. Set
$$\delta=\inf\{\|x_n\|_{(L_p+L_2)(\mathcal{M})}:n\geq1\}.$$
We claim that $\delta>0$.
Assume that $\liminf_{n\to \infty}\|x_n\|_{(L_p+L_2)(\mathcal{M})}=0$. Then there exists a subsequence $(x_{n_k})_{k\geq1}$ which converges to zero in  $(L_p+L_2)(\mathcal{M})$, and consequently, $(x_{n_k})_{k\geq1}$ converges to zero in measure (see e.g. \cite[Lemma 2.4]{HLS2017} or \cite[Lemma 4.4]{DDP1989}). By the second assumption of the theorem, the martingale difference sequence $(x_n)_{n\geq1}$ is  $E$-equi-integrable in $E(\mathcal{M})$. Then it follows from Lemma \ref{e imply com} that $\lim_{k\to \infty}\|x_{n_k}\|_{E(\mathcal{M})}=0$, which contradicts our initial assumption. Hence $\delta>0$.
Choose a square-summable sequence $(\lambda_n)_{n\geq1}$ of scalars such that $\sum_{k\geq1}\lambda_nx_n\in E(\mathcal{M})$. Then we have
\begin{align*}\label{t1}
\delta \Big(\sum_{n\geq1}|\lambda_n|^2\Big)^{1/2}&\leq \Big(\sum_{n\geq1}\|\lambda_nx_n\|_{(L_p+L_2)(\mathcal{M})}^2\Big)^{1/2}.
\end{align*}
By Lemma \ref{martingale lemma}, we have
$$\Big(\sum_{n\geq1}|\lambda_n|^2\Big)^{1/2}\lesssim \Big\|\sum_{n\geq1}\lambda_nx_n\Big\|_{(L_1+L_2)(\mathcal{M})}\leq \Big\|\sum_{n\geq1}\lambda_nx_n\Big\|_{E(\mathcal{M})}.$$
This means the martingale difference sequence $(x_n)_{n\geq1}$ is $2$-co-lacunary  in $E(\mathcal{M})$ and the proof is complete.
\end{proof}

\noindent{\bf Acknowledgement.}
Most of the work was completed when the second author was visiting UNSW. He would like to express his gratitude to the School of Mathematics and Statistics of UNSW for its warm hospitality.
The authors would like to thank Dmitriy Zanin for useful communication.  The authors would  like to thank Thomas Scheckter for his careful reading and helpful comments that significantly improve the presentation of the paper. The authors would also like to thank the anonymous reviewer who suggested the current proofs of  Lemmas \ref{CR lemma} and  \ref{cor-2-concave} which are substantially shorter than their initial versions.

%

\end{document}